\documentclass[a4paper]{amsart}
\usepackage[foot]{amsaddr}
\usepackage[T1]{fontenc}
\usepackage{fullpage}
\usepackage{amsmath,amsthm,amssymb,amscd}
\usepackage[dvipsnames]{xcolor}
\definecolor{darkgreen}{rgb}{0,0.5,0}
\definecolor{darkblue}{rgb}{0,0,0.5}
\usepackage[colorlinks=true,urlcolor=darkblue,citecolor=darkgreen,linkcolor=darkblue]{hyperref}
\usepackage{graphicx}
\usepackage{lmodern}
\usepackage{tabularray}
\usepackage{adjustbox}
\usepackage{tikz}
\usetikzlibrary{graphs, quotes}

\newcommand{\bbN}{\mathbb{N}}

\newcommand{\calP}{\mathcal{P}}

\usepackage[sc]{mathpazo}

\newtheorem{thm}{Theorem}[section]
\newtheorem{prop}[thm]{Proposition}
\newtheorem{lem}[thm]{Lemma}
\newtheorem{cor}[thm]{Corollary}
\newtheorem{conjecture}[thm]{Conjecture}

\theoremstyle{remark}
\newtheorem{remx}[thm]{Remark}
\newtheorem{examplex}[thm]{Example}

\theoremstyle{definition}
\newtheorem{question}[thm]{Question}

\newcommand{\qedblack}{\hfill \scalebox{0.7}{\ensuremath{\blacksquare}}}
\newenvironment{example}
  {\pushQED{\qedblack}\begin{examplex}}
  {\popQED\end{examplex}}

\newenvironment{rem}
  {\pushQED{\qedblack}\begin{remx}}
  {\popQED\end{remx}}

\title{On the clique covering numbers of Johnson graphs}

\author{S\o ren Fuglede J\o rgensen}
  \address{Kvantify, Rosenvængets Allé 25, DK-2100 Copenhagen, Denmark}
  \email{sfj@kvantify.dk}

\begin{document}

\begin{abstract}
    We initiate a study of the vertex clique covering numbers of Johnson graphs $J(N, k)$, the smallest numbers of cliques necessary to cover the vertices of those graphs.

    We prove identities for the values of these numbers when $k \leq 3$, and $k \geq N - 3$, and using computational methods, we provide explicit values for a range of small graphs. By drawing on connections to coding theory and combinatorial design theory, we prove various bounds on the clique covering numbers for general Johnson graphs, and we show how constant-weight lexicodes can be utilized to create optimal covers of $J(2k, k)$ when $k$ is a small power of two.
\end{abstract}

\maketitle

\tableofcontents

\section{Introduction}

A clique in an undirected graph $G = (V, E)$ is a fully-connected subgraph, and a (vertex) clique cover is a collection of cliques $\{G_i = (V_i, E_i)\}$ for which $\bigcup_i V_i = V$. A maximal clique is a clique which is not a proper subset of any other clique, and a maximum clique is a clique whose cardinality is maximal among all cliques. It is clear that if there is a clique cover of a given cardinality, there is also a clique cover of the same cardinality (or smaller) consisting of only maximal cliques.

The clique covering number $\theta(G)$ of $G$ is the cardinality of a smallest clique cover.

Here, we investigate the clique covering numbers of Johnson graphs, defined as follows:

For $N \in \bbN$, let $[N] = \{1, \dots, N\}$, and for $k \in \bbN$, $0 < k < N$, let
\[
    \calP_k([N]) = \{S \in \calP([N]) \mid \lvert S \rvert = k \}
\]
be the $k$-element subsets of $[N]$. The Johnson graph $J(N, k)$ is the graph whose vertices are these subsets, $V(J(N, k)) = \calP_k([N])$, and whose edges are
\[
    E(J(N, k)) = \{(S_1, S_2) \in V(J(N, k)) \times V(J(N, k)) \mid \lvert S_1 \cap S_2 \rvert = k - 1 \}.
\]

While Johnson graphs are ubiquitous in mathematics, and in particular in coding theory and combinatorial design theory, our motivation for studying their structure is the fact that they form the solution space for a wide range of optimization problems. These include cardinality constrained combinatorial optimization problems -- problems whose objectives are maps $V(J(N, k)) \to \mathbb{R}$ and for which natural optimization algorithms are local search algorithms that may be described as walks on $J(N, k)$ -- but also problems in quantum chemistry, such as electronic structure problems involving molecules with $N$ available spin orbitals out of which $k$ are occupied. Then, the space of solutions is the complex span of the vertices, $\mathcal{H} = \mathrm{span}_{\mathbb{C}}V(J(N, k))$, the problem is to minimize an energy expectation $\mathcal{H} \to \mathbb{R}$, and the edge structure features in algorithms from coupled cluster theory based on using so-called single excitation operators to search the space. In this context, the vertices of the Johnson graph are often represented as length $N$ bit strings with $k$ total $1$s, so that e.g. the string 1010010 would correspond to what we will write $\{1, 3, 5\}$, and two bit strings are adjacent if their Hamming distance is $2$.

In Section~\ref{sec:exact-values}, we will find exact values for $\theta(J(N, k))$ for $k \leq 3$ and $k \geq N - 3$, and we will provide exact values of $\theta(J(N, k))$ for various low values of $N$ and $k$. In Section~\ref{sec:recursive-bounds}, we obtain bounds on $\theta(J(N, k))$ from recursive constructions. In Section~\ref{sec:coding-theory}, we will see how constant-weight codes can be used to construct collections of cliques in the Johnson graphs and thereby relate known bounds on the number of codewords in such codes to the values of $\theta(J(N, k))$, and we find that for $k = 1, 2, 4, 6, 8, 16$, the value of $\theta(J(2k, k))$ is the $k$'th Catalan number. The construction of constant-weight codes is, on the other hand, closely related to the construction of combinatorial designs, and in Section~\ref{sec:design-theory}, we will see how this relation offers additional bounds on $\theta(J(N, k))$ and see how the existence of Steiner systems implies that the clique covering numbers of corresponding Johnson graphs are in a sense as small as they can get.

\textbf{Acknowledgements.} The author thanks Patrick Ettenhuber for introducing the problem and for explaining the quantum chemistry context in which it came up, and Janus Halleløv Wesenberg for useful comments on an earlier draft.

\section{Exact values of the clique covering numbers}
\label{sec:exact-values}
The maximal cliques of $J(N, k)$ may be classified as follows, adapted from \cite{johnsonauto,shuldiner2022cliquestructurejohnsongraphs}:

\begin{thm}
    \label{thm:classification}
    Let $N > 2$. The vertices of a maximal clique in $J(N, k)$ are either of the form
    \begin{itemize}
        \item $A_S^N = \{ S \cup \{ x \} \mid x \in [N] \setminus S \}$ for an $S \in \calP_{k-1}([N])$, $k < N - 1$, or
        \item $B_S^N = \{ S \setminus \{x\} \mid x \in S\}$ for an $S \in \calP_{k+1}([N])$, $k > 1$.
    \end{itemize}
    Moreover these sets are distinct maximal cliques. For $N = 2$, $k = 1$, the only maximal clique is the graph itself.
\end{thm}

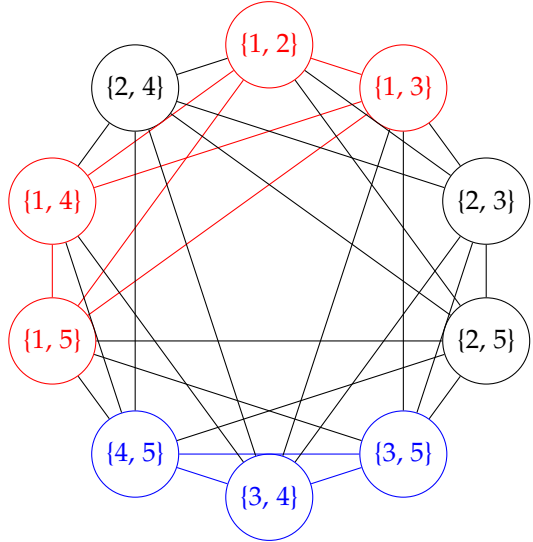
\begin{figure}[ht]
    \begin{tikzpicture}
    \node[draw, circle, red] (12) at (90:3) {\{1, 2\}};
    \node[draw, circle, red] (13) at (54:3) {\{1, 3\}};
    \node[draw, circle, red] (14) at (-198:3) {\{1, 4\}};
    \node[draw, circle, red] (15) at (-162:3) {\{1, 5\}};
    \node[draw, circle] (23) at (18:3) {\{2, 3\}};
    \node[draw, circle] (24) at (-234:3) {\{2, 4\}};
    \node[draw, circle] (25) at (-18:3) {\{2, 5\}};
    \node[draw, circle, blue] (34) at (-90:3) {\{3, 4\}};
    \node[draw, circle, blue] (35) at (-54:3) {\{3, 5\}};
    \node[draw, circle, blue] (45) at (-126:3) {\{4, 5\}};

    \foreach \i/\j in {12/13, 12/14, 12/15, 13/14, 13/15, 14/15} {
        \draw[red] (\i) -- (\j);
    }

    \foreach \i/\j in {34/35, 34/45, 35/45} {
        \draw[blue] (\i) -- (\j);
    }

    \foreach \i/\j in {12/23, 12/24, 12/25, 13/23, 13/34, 13/35, 14/24,
                         14/34, 14/45, 15/35, 15/25, 15/45, 23/24, 23/34,
                         23/35, 23/25, 24/25, 24/34, 24/45, 25/35, 25/45} {
        \draw (\i) -- (\j);
    }
\end{tikzpicture}
    \caption{The graph $J(5, 2)$. The clique $A_{\{1\}}^5$ is highlighted in red, and the clique $B_{\{3,4,5\}}^5$ is highlighted in blue.}
    \label{fig:J52}
\end{figure}

That the edge case conditions on $k$ are necessary can be seen by verifying the case $N = 3$, $k \in \{1, 2\}$. In \cite{shuldiner2022cliquestructurejohnsongraphs}, it is shown that every edge in $J(N, k)$ belongs to exactly one of these maximal cliques, so that in particular, the edge clique covering number, i.e. the cardinality of the smallest clique cover covering all edges of the graph, is simply the number of different maximal cliques. Here, the focus is on the vertex clique covering number, which is generally smaller than the edge clique covering number.

We will refer to a clique of the form $A_S^N$ as being of \emph{type A}, and a clique of the form $B_S^N$ as being of \emph{type B}. Similarly, we will say that a clique cover will be of \emph{type A} (resp. B) if all of its elements are of type A (resp. B). If a clique cover contains elements of both types, we will call it \emph{mixed}. See Figure~\ref{fig:J52} for an example of the two types of cliques in $J(5, 2)$.

In particular, one may consider the problem of finding $\theta(J(N, k))$ for a specific Johnson graph as an instance of the set cover problem. Fixed instances can be solved to optimality for small values of $N$ and $k$ with off-the-shelf integer linear programming  (ILP) solvers. Heuristic local search solvers can also provide upper bounds, in some cases faster than ILP solvers can, and in Table~\ref{tab:exact} we show the lower bounds obtained through ILP solvers, and the best upper bounds obtained from the combination of ILP solvers with a local search algorithm based on simulated annealing.

\begin{table}[h]
    \centering
    \begin{tabular}{r|rrrrrrrrrrrrrr}
        $N$ & $k=1$ & 2&3&4&5&6&7&8&9&10&11&12&13&14 \\
        \hline
        2 & 1 & &&&&&& \\
        3 & 1 & 1 &&&&&& \\
        4 & 1 & 2 & 1 &&&&& \\
        5 & 1 & 3 & 3 & 1 &&&& \\
        6 & 1 & 4 & 6 & 4 & 1 &&& \\
        7 & 1 & 5 & 9 & 9 & 5 & 1 && \\
        8 & 1 & 6 & 12 & 14 & 12 & 6 & 1 & \\
        9 & 1 & 7 & 16 & 25 & 25 & 16 & 7 & 1 \\
        10 & 1 & 8 & 20 & 40 & 46 & 40 & 20 & 8 & 1  \\
        11 & 1 & 9 & 25 & 56 & \scriptsize{74--77} &\scriptsize{74--77}&56& 25 & 9 & 1 \\
        12 & 1 & 10 & 30 &  \scriptsize{71--76} &\scriptsize{110--115}&132&\scriptsize{110--115}& \scriptsize{71--76} & 30 & 10 & 1 \\
        13 & 1 & 11 & 36 &  \scriptsize{92-110}&\scriptsize{154-185}&\scriptsize{224-280}&\scriptsize{224-280}&\scriptsize{154-185}&\scriptsize{92-110}& 36 & 11 & 1 \\
        14 & 1 & 12 & 42 &  \scriptsize{117-141}&\scriptsize{211-259}&\scriptsize{349-476}& \scriptsize{429-588} &\scriptsize{349-476}&\scriptsize{211-259}&\scriptsize{117-141}& 42 & 12 & 1\\
        15 & 1 & 13 & 49 &  \scriptsize{146-182}&\scriptsize{283-363}&\scriptsize{513-816}&\scriptsize{716-1110}&\scriptsize{716-1110}&\scriptsize{513-816}&\scriptsize{283-363}&\scriptsize{146-182}& 49 & 13 & 1
    \end{tabular}
    \bigskip
    \caption{Values of $\theta(J(N, k))$ for small $N$ and $k$, as obtained through the combination of an integer programming solver and a local search algorithm based on simulated annealing. For $N \geq 11$ we report intervals of possible values, e.g. $74 \leq \theta(J(11, 5)) \leq 77$. Obtaining any one of these bounds takes a few hours on a 64 core Intel(R) Xeon(R) Platinum 8272CL CPU @ 2.60GHz running the Gurobi parallel ILP solver on the natural ILP formulation of set cover, without taking into account any of the symmetries of the problem, with the exception of $N = 11$, $k = 4$, which ran to completion in 13 hours.}
    \label{tab:exact}
\end{table}

From here we can already read off a few patterns; note first that $\theta(J(N, k)) = \theta(J(N, N - k))$ for all $N$ and $k$ since the graphs are isomorphic: A set of maximal cliques covering $J(N, k)$ can be turned into a cover of $J(N, N-k)$ by replacing $A_S$ with $B_{S^c}$ and vice versa.

\begin{prop}
    We have $\theta(J(N, 1)) = 1$ for all $N > 1$.
\end{prop}
\begin{proof}
    The vertices of the graph are singletons, and $\{i\} \cap \{j \} = 0$ whenever $i \not= j$, so the graph is complete.
\end{proof}

In the case $k = 2$, note that a minimal clique cover of type A is $\{A_{\{1\}}^N, \dots, A_{\{N-1\}}^N\}$, but the result below shows that there are smaller mixed covers.

\begin{thm}
    \label{thm:JN2}
    We have $\theta(J(N, 2)) = N - 2$ for all $N > 2$.
\end{thm}
\begin{proof}
    We claim that
    \[
        \{A_{\{1\}}^N, \dots, A_{\{N-3\}}^N, B_{\{N-2,N-1,N\}}^N\}
    \]
    is a minimal clique cover. Let $\{n_1, n_2\} \in V(J(N, 2))$, and assume that $n_1 < n_2$. If $n_1 \leq N - 3$, then $\{n_1, n_2\} \in A_{\{n_1\}}^N$. Otherwise, $\{n_1, n_2\} \subseteq \{N - 2, N - 1, N\}$, so $\{n_1, n_2\} \in B_{\{N-2,N-1,N\}}^N$, so the set is a clique cover.

    To see that no smaller cover exists, suppose that a given clique cover contains $n_A$ cliques of the type A in, and $n_B$ of type B. Assume that $0 \leq n_A \leq N$. As such, there are $m = N - n_A$ elements of $[N]$ whose corresponding singleton sets are not among the singletons defining the $n_A$ cliques. These $m$ elements correspond to $\binom{m}{2}$ vertices in $J(N, 2)$ that must then be covered by the $n_B$ remaining cliques. Each of these cliques cover at most $3$ elements, so we need at least $\left\lceil \binom{m}{2} / 3 \right\rceil$ of them. Put together, we may bound the total number of cliques by
    \[
        n_A + n_B \geq N - m + \left\lceil \binom{m}{2}/3 \right\rceil,
    \]
    so we need at least $N - 2$ for all $m \geq 0$.
\end{proof}

Similarly, for the case $k = 3$, we can explicitly construct a minimal clique cover as follows: The idea is to partition $[N]$ into two parts such that every $3$-element set by the pigeon hole principle has a $2$-element overlap with those two parts. Concretely, let $N > 3$ and define $S_N \subseteq \calP_2([N])$ by
\[
    S_N = \{\{i, j\} \subseteq [N] \mid i \not= j, (\{i, j\} \subseteq \{1, \dots, \lfloor \tfrac{N}{2} \rfloor\} \mbox{ or } \{i, j\} \subseteq \{\lfloor \tfrac{N}{2} \rfloor + 1, \dots, N \} )\}.
\]
This defines a set $\mathcal{C}_N$ of maximal cliques in $J(N, 3)$ by
\[
    \mathcal{C}_N = \{ A_S^N \mid S \in S_N \}.
\]
That this is a minimal clique cover among all clique covers of type A is known and can be seen as a consequence of Mantel's theorem, but we need to be a bit more careful to make sure that no smaller mixed covers exist.

\begin{lem}
    For any $N > 3$, the cardinality of $\mathcal{C}_N$, is $\lvert \mathcal{C}_N \rvert = \lvert S_N \rvert = \left\lfloor \tfrac{(N-1)^2}{4} \right\rfloor.$
\end{lem}
\begin{proof}
    If $N$ is even, $N = 2m$, then
    \[
        \left\lfloor \frac{(N-1)^2}{4} \right\rfloor = \left\lfloor \frac{N^2}{4} - \frac{N}{2} + \frac{1}{4}\right\rfloor = m^2 - m = \binom{m}{2}+\binom{m}{2} = \lvert S_N \rvert,
    \]
    and if $N$ is odd, $N = 2m + 1$, then
    \[
        \left\lfloor \frac{(N-1)^2}{4} \right\rfloor = \frac{(2m)^2}{4} = \frac{m(m-1)}{2} + \frac{m(m+1)}{2} = \binom{m}{2} + \binom{m+1}{2} = \lvert S_N \rvert.
    \]
\end{proof}
The following result is often referred to as Goodman's bound, following \cite{Goodman1959-gm}.
\begin{thm}
    \label{thm:goodmandbound}
    In any graph with $n$ vertices and $m$ edges, there are at least
    \[
        \frac{4m}{3n} \left(m - \frac{n^2}{4}\right)
    \]
    triangles.
\end{thm}
\begin{thm}
    \label{thm:JN3}
    For any $N > 3$, the set $\mathcal{C}_N$ defined above is a clique cover of $J(N, 3)$. For $N > 7$ it is a minimal clique cover.
\end{thm}
\begin{proof}
    Let $\{i, j, k\} \in V(J(N, 3))$. Two of these three elements, say $i$ and $j$, will form a subset $S = \{i, j\}$ of either $ \{1, \dots, \lfloor \tfrac{N}{2} \rfloor\}$ or $\{\lfloor \tfrac{N}{2} \rfloor + 1, \dots, N \}$, and therefore belong to $S_N$, so $\{i, j, k\} = S \cup \{k\}$ is an element of the maximal clique $A_S^N \in \mathcal{C}_N$.

    To show that this construction is optimal for $N > 7$, consider a collection containing $n_a$ and $n_b$ cliques of types A and B respectively, consisting of $n_a + n_b = \left\lfloor \tfrac{(N-1)^2}{4} \right\rfloor - 1$ cliques in total, and let us bound from above the number of elements of $V(J(N, 3))$ that this can cover.

    First, we will ignore cliques of type B and provide an upper bound on the number of vertices that can be covered by a union of cliques of type A. Each clique of type A is defined by a 2-element subset of $[N]$ and can be identified with an edge in the complete graph $K_N$ on $N$ vertices; let $E_A$ denote the set of $n_A$ edges obtained from the cover. From this perspective, an element $\{i, j, k\} \in V(J(N, 3))$ is covered by the clique if the corresponding triangle in $K_N$ contains at least one of the edges in $E_A$. On the other hand, an element of $V(J(N, 3))$ is \emph{not} covered if the edges of its corresponding triangle are in the complement of $E_A$. The maximum number of elements in $V(J(N, 3))$ that can be covered by the cliques can therefore be bounded from above by $\binom{N}{3} - t$, where $t$ is the minimum number of triangles in any graph on $N$ vertices with $m = \lvert E(K_N) \rvert - \lvert E_A \rvert = N(N-1)/2 - n_a$ edges.

    Since each clique of type B can cover at most $4$ elements of $V(J(N, 3))$, by Theorem~\ref{thm:goodmandbound}, an immediate upper bound on the total number of elements that are covered is therefore
    \begin{align*}
        c(n_a) &= \binom{N}{3} - \frac{4m}{3N}\left(m - \frac{N^2}{4}\right) + 4n_b \\
          &= \binom{N}{3} - \frac{2N(N-1) - 4n_a}{3N} \left(\frac{N(N-1)}{2} - n_a - \frac{N^2}{4}\right) + 4\left(\left\lfloor \frac{(N-1)^2}{4} \right\rfloor -1 - n_a\right).
    \end{align*}
    For fixed $N$, we may treat this is a quadratic polynomial in $n_a$, the derivative of which is
    \begin{align*}
        c'(n_a) &= - 2 \frac{4}{3N} n_a +\frac{2N(N-1)}{3N} +\frac{4N(N-1)}{6N} - \frac{4}{3N} \frac{N^2}{4} - 4 \\
         &= -\frac{8}{3N} n_a + N - \frac{16}{3}.
    \end{align*}
    so $c$ is maximized when $n_a$ is obtained by rounding $\widetilde{n_a} = \min\left(\frac{3}{8}N^2 - 2N, \left\lfloor \tfrac{(N-1)^2}{4}\right\rfloor - 1\right)$.

    When $N \geq 12$, this means that $n_a = \left\lfloor \tfrac{(N-1)^2}{4} \right\rfloor - 1$ and $n_b = 0$, which completes the proof for all but finitely many cases.

    For $8 \leq N \leq 11$, we check by hand that the value of $c$ is lower than $\binom{N}{3}$ at the candidate values of $n_a$; see Table~\ref{tab:upperbound}.

    \begin{table}
        \begin{tblr}{c|c|c|c}
            $N$ & $n_a$ & $\lceil c(n_a) \rceil$ & $\binom{N}{3}$ \\
            \hline
            $8$ & $8$ & $55$ & $56$ \\
            $9$ & $12$ & $83$ & $84$ \\
            $9$ & $13$ & $83$ & $84$ \\
            $10$ & $17$ & $117$ & $120$ \\
            $10$ & $18$ & $117$ & $120$ \\
            $11$ & $23$ & $163$ & $165$ \\
            $11$ & $24$ & $163$ & $165$ \\
        \end{tblr}
        \caption{Upper bounds on the number of coverable elements for $8 \leq N \leq 11$.}
        \label{tab:upperbound}
    \end{table}

\end{proof}
This leaves the cases $N \leq 7$ for which the bound in the proof fails to rule out the usefulness of cliques of type B. Specifically, $\lceil c(n_a) \rceil \geq \binom{N}{3}$ when $N = 6$, $n_a \leq 4$, and when $N = 7$, $n_a \leq 7$. Instead, values of $\theta(J(N, 3))$ for $N \leq 7$ can be read off directly from Table~\ref{tab:exact}, but are also small enough that one can prove optimality computationally through exhaustive enumeration.

Summarizing, this gives the following values for $\theta(J(N, 3))$:
\begin{thm}
    \label{thm:J-N-3}
    For $N > 3$ we have,
    \begin{align*}
        \theta(J(N, 3)) = \begin{cases}
            1,& \text{if $N = 4$},\\
            3,& \text{if $N = 5$},\\
            \left\lfloor \frac{(N-1)^2}{4} \right\rfloor,& \text{if $N > 5$}.\\
        \end{cases}
    \end{align*}
\end{thm}
For $k = 4$, we have no explicit constructions of minimal clique covers for all $N$, and given the connection to design theory which we will explore below, finding such constructions might be a lot to hope for. In Table~\ref{tab:smallsolutions}, we provide minimal clique covers for $k \geq 4$ and $N \leq 11$.

\section{Recursive bounds}
\label{sec:recursive-bounds}
One may note that the first few rows of Table~\ref{tab:exact} agree with the familiar Pascal's triangle of binomial numbers but after a while grows more slowly. This suggests that $\theta(J(N, k)) \leq \binom{N-2}{k-1}$, which we prove below.

\begin{prop}
    \label{prop:recursion}
    For $N > 2$ and $k > 1$ we have
    \[
        \theta(J(N, k)) \leq \theta(J(N - 1, k - 1)) + \theta(J(N - 1, k)).
    \]
\end{prop}
\begin{proof}
    The case $N = 3$ can be verified by hand, allowing us to assume that $N > 3$ to avoid having to deal the slightly complicating special cases in Theorem~\ref{thm:classification}.

    Let $\{S_1, \dots, S_n\}$ be a minimal clique cover of $J(N-1, k-1)$ and $\{T_1, \dots, T_m\}$ be a minimal clique cover of $J(N-1,k)$. Assume without loss of generality that all $S_i$ and $T_j$ are maximal cliques in their respective graphs.

    For a given $i \in [n]$, by Theorem~\ref{thm:classification}, we either have $S_i = A_S^{N-1}$ for an $S \in \calP_{k-2}([N-1])$, or $S_i = B_S^{N-1}$ for an $S \in \calP_k([N-1])$. In the first case, define
    \[
        \tilde{S}_i = A_{S \cup \{N\}}^N,
    \]
    and in the second case define
    \[
        \tilde{S}_i = B_{S \cup \{N\}}^N.
    \]
    In either case, $\tilde{S}_i$ is a maximal clique in $J(N,k)$.
    
    Similarly, for each $T_j$, either there exists $S \in \calP_{k-1}([N-1])$ such that $T_j = A_S^{N-1}$ or there exists an $S \in \calP_{k+1}([N-1])$ such that $T_j = B_S^{N-1}$. In the first case, define
    \[
        \tilde{T}_j = A_S^N
    \]
    and in the second case simply define $\tilde{T}_j = B_S^N = T_j$. We claim that
    \[
        \{ \tilde{S}_1, \dots, \tilde{S}_n, \tilde{T}_1, \dots, \tilde{T}_m \}
    \]
    is a clique cover of $J(N, k)$, from which it follows that
    \[
        \theta(J(N, k)) \leq n + m = \theta(J(N - 1, k - 1)) + \theta(J(N - 1, k)).
    \]
    That the sets are cliques is also part of Theorem~\ref{thm:classification}. To see that this is a covering, let $S \in \calP_k([N])$ be a vertex. If $N \notin S$, then $S \in T_j$ for some $j$ since the $T_j$ cover $J(N-1, k)$. Since $T_j \subseteq \tilde{T_j}$, then $S$ is covered. If $N \in S$, there exists an $i$ such that $S \setminus \{N\} \in S_i$, so $S \in \tilde{S}_i$.
\end{proof}

\begin{thm}
    \label{thm:recursivebound}
    For any $N$ and $k$ we have
    \[
        \theta(J(N, k)) \leq \binom{N-2}{k-1}.
    \]
\end{thm}
\begin{proof}
    The claim holds for $(N, k) = (2, 1)$. Let $N > 2$, let $k \in \bbN$, $0 < k < N$ and assume that
    \[
        \theta(J(M, l)) \leq \binom{M-2}{l-1}
    \]
    for $M < N$ and all $l$, $1 \leq l \leq M$. Then by Proposition~\ref{prop:recursion},
    \[
        \theta(J(N, k)) \leq \theta(J(N - 1, k - 1)) + \theta(J(N -1, k)) \leq \binom{N - 3}{k - 2} + \binom{N - 3}{k - 1} = \binom{N - 2}{k - 1}.
    \]
\end{proof}
Note that Proposition~\ref{prop:recursion} also provides a recursive construction of a clique cover of $J(N, k)$ of cardinality $\binom{N-2}{k-1}$.

As one may read from Table~\ref{tab:exact}, the lowest value of $N$ for which the bound fails to be tight is $N = 7$, at $k = 3$. Here, the construction in Proposition~\ref{prop:recursion} gives the cardinality $10$ clique cover
\begin{gather*}
\left\{A_{\{1,2\}}, A_{\{1,3\}},A_{\{1,4\}},A_{\{2,3\}},A_{\{2,4\}},A_{\{3,4\}}, B_{\{1,5,6,7\}},B_{\{2,5,6,7\}},B_{\{3,5,6,7\}},B_{\{4,5,6,7\}}\right\}.
\end{gather*}
Note that removing any clique from this collection will result in a collection that no longer covers $J(7, 3)$. An example of a cardinality $9$ minimal clique cover is the one obtained from the construction in Theorem~\ref{thm:JN3},
\begin{gather*}
\left\{A_{\{1,2\}},A_{\{1,3\}},A_{\{1,4\}},A_{\{2,3\}},A_{\{2,4\}},A_{\{3,4\}},A_{\{5,6\}},A_{\{5,7\}},A_{\{6,7\}}\right\}.
\end{gather*}

\section{Bounds from coding theory}
\label{sec:coding-theory}
Next up, we will see how we may draw on known results on constant-weight codes. Let 
\[
    \omega(J(N,k)) = \max(N - k + 1, k + 1)
\]
denote the clique number of $J(N, k)$, i.e. the size of a maximum clique, and let $C_k = \tfrac{1}{k+1} \binom{2k}{k}$ denote the $k$'th Catalan number. A simple lower bound on $\theta(J(N, k))$ is obtained by noting that we need at least as many cliques as we would we had a cover by disjoint maximum cliques:
\begin{prop}
    \label{prop:simple-lower-bound}
    For all $N$ and $k$ we have
    \[
        \theta(J(N, k)) \geq \frac{1}{\omega(J(N, k))} \binom{N}{k}.
    \]
    In particular,
    \[
        \theta(J(2k, k)) \geq C_k.
    \]
\end{prop}

Only in rare cases is this bound tight. For the values shown in Table~\ref{tab:smallsolutions}, we find that the bound is realized only when $k = 1$, $k = N - 1$, $N \leq 5$, or when $(N, k)$ is either $(8, 4)$ or $(12, 6)$.

For any graph $G$, we have $\theta(G) \geq \alpha(G)$, where $\alpha(G)$ is the independence number of $G$, i.e. the cardinality of a largest independent set in $G$. An independent set in $J(N, k)$ is known as a binary constant-weight code of length $N$, weight $k$, and Hamming distance at least $4$. The maximum number of elements, codewords, in such a code is denoted $A(N, 4, k) = \alpha(J(N, k))$ and has been studied to a great extent. For instance, we may note that by comparing the values in Table~\ref{tab:smallsolutions} to the values of $A(N, 4, k)$ provided in \cite{ostergard2010}, the only non-trivial cases of where $\theta(J(N, k)) = \alpha(J(N, k))$ that we know of are those where $(N, k)$ is either $(8, 4)$ or $(12, 6)$.

In fact, in these cases, the codes offer a construction of corresponding collections of disjoint maximal cliques, associating to each codeword a maximal clique, so that when the resulting collection covers $V(J(N, k))$, it is necessarily a minimal cover: Let $j \in [N]$ and define for a codeword $S \in V(J(N, k))$ the clique
\[
    C_S^j = \begin{cases}
        A_{S \setminus \{ j \}}^N & \text{if $j \in S$,} \\
        B_{S \cup \{j \}}^N & \text {otherwise.}
    \end{cases}
\]

\begin{prop}
    \label{prop:indep-set-cover}
    For an independent set $\mathcal{W} \subseteq V(J(N, k))$, the maximal cliques $\{ C_S^j \}_{S \in \mathcal{W}}$ are pairwise disjoint. In particular, if
    \[
        \alpha(J(N, k)) = \frac{1}{\omega(J(N, k))} \binom{N}{k},
    \]
    and $\lvert C_S^j \rvert = \omega(J(N, k))$ for all $S \in \mathcal{W}$, then the collection is a minimal cover and $\alpha(J(N, k)) = \theta(J(N, k))$.
\end{prop}
\begin{proof}
    Let $S_1, S_2 \in \mathcal{W}$. We will consider three different cases. Suppose first that $j \in S_1 \cap S_2$, and let $S \in C_{S_1}^j \cap C_{S_2}^j = A_{S_1 \setminus \{ j \}}^N \cap A_{S_2 \setminus \{ j \}}^N$. Then there exist $x \in [N] \setminus (S_1 \setminus \{j\})$ and $y \in [N] \setminus (S_2 \setminus \{j\})$ so that
    \[
        S = (S_1 \setminus \{j\}) \cup \{x\} = (S_2 \setminus \{j\}) \cup \{y\}. 
    \]
    It follows that $\lvert S_1 \cap S_2 \rvert \geq k - 1$, so since $\mathcal{W}$ was an independent set, $S_1 = S_2$.

    Now suppose that $j \in S_1$, $j \not\in S_2$, and let $S \in C_{S_1}^j \cap C_{S_2}^j$. Then there exist $x \in [N] \setminus (S_1 \setminus \{j\})$ and $y \in S_2 \cup \{j\}$ so that
    \[
        S = (S_1 \setminus \{j\}) \cup \{x\} = (S_2 \cup \{j\}) \setminus \{y\},
    \]
    and as above, $\lvert S_1 \cap S_2 \rvert \geq k - 1$, so $S_1 = S_2$, which contradicts that $j \in S_1 \setminus S_2$.

    Lastly, suppose that $j \not\in S_1$, $j \not\in S_2$, so that
    \[
        S = (S_1 \cup \{j\}) \setminus \{x\} = (S_2 \cup \{j\}) \setminus \{y\}
    \]
    for some $x \in S_1 \cup \{j\}$, $y \in S_2 \cup \{j\}$. Once more, this is only possible if $\lvert S_1 \cap S_2 \rvert \geq k - 1$, so $S_1 = S_2$.
\end{proof}
\begin{cor}
    If $\alpha(J(2k, k)) = C_k$, then $\theta(J(2k, k)) = C_k$.
\end{cor}
Since $\alpha(J(2k, k)) = C_k$ for $k = 4, 6$, this gives us a different way of seeing that $\theta(J(8, 4)) = 14$ and $\theta(J(12, 6)) = 132$. Note that those values were already captured in Table~\ref{tab:smallsolutions}.

\begin{prop}
    If $\alpha(J(N, k)) = \theta(J(N, k))$, then $J(N, k)$ has a minimal clique cover by disjoint maximal cliques.
\end{prop}
\begin{proof}
    The clique-coclique bound for vertex-transitive graphs, cf. e.g. \cite[Cor.~2.1.2]{erdos-ko-rado-book}, implies that
    \begin{align}
        \label{eq:clique-coclique}
        \alpha(J(N, k)) \leq \frac{1}{\omega(J(N, k))} \binom{N}{k}.
    \end{align}
\end{proof}

As mentioned above, the only non-trivial cases we know of for which $\alpha(J(N, k)) = \theta(J(N, k))$ are $N = 8$, $k = 4$ and $N = 12$, $k = 6$. One might speculate if this pattern continues for other cases of $N = 2k$, where $k$ is even, but it follows from the Johnson bounds \cite{Zeger2000-cm, johnsonbound} that e.g. $\alpha(J(16, 8)) \leq 1280 < 1430 = C_8$. The proposition below allows us to rule out a further collection of cases.

\begin{prop}
    \label{prop:alpha-prime}
    Assume that $\alpha(J(2k, k)) = C_k$. Then $k + 1$ is prime.
\end{prop}
\begin{proof}
    In general, the Johnson bounds tell us that
    \begin{align*}
        \alpha(J(N, k)) \leq \left\lfloor \frac{N}{k} \alpha(J(N-1, k-1)) \right\rfloor, \\
        \alpha(J(N, k)) \leq \left\lfloor \frac{N}{N-k} \alpha(J(N-1, k)) \right\rfloor.
    \end{align*}
    for $k \in \{1, \dots, N - 1\}$. It is known that $\alpha(J(n, 2)) = \lfloor\tfrac{n}{2} \rfloor$ for all $n > 1$, so by applying the first Johnson bound $k - 2$ times, we find that
    \begin{align}
        \label{eq:johnson-recursion}
        \alpha(J(2k, k)) \leq \left\lfloor \frac{2k}{k} \left\lfloor \frac{2k-1}{k-1} \left\lfloor \frac{2k-2}{k-2} \cdots \left\lfloor \frac{2k-(k-3)}{k-(k-3)} \left\lfloor \frac{2k-(k-2)}{k-(k-2)}\right\rfloor\right\rfloor\right\rfloor\right\rfloor\right\rfloor.
    \end{align}
    On the other hand, a well-known recursion for the Catalan numbers tell us that
    \[
        C_k = \frac{2k}{k} \frac{2k-1}{k-1} \frac{2k-2}{k-2}\cdots \frac{2k-(k-3)}{k-(k-3)}\frac{2k-(k-2)}{k-(k-2)}.
    \]
    In particular, $\alpha(J(2k, k)) = C_k$ if and only if we never round down in \eqref{eq:johnson-recursion}, which happens exactly when $k + 1$ is prime.
\end{proof}
\begin{example}
    Note than other known values of $A(2k, 4, k)$ allow us to further rule out the possibility that $\alpha(J(2k, k)) = C_k$. For example, from \cite[Table~I]{Zeger2000-cm}, we read that $\alpha(J(20, 10)) \leq 16652 < 16796 = C_{10}$ and $\alpha(J(24, 12)) \leq 207078 < 208012 = C_{12}$, which lets us rule out $k = 10$ and $k = 12$, but all other values of $k$ that are one less than a prime are left open.
\end{example}

We can also note that inequality~\eqref{eq:clique-coclique} being tight is not by itself sufficient to guarantee that we have $\alpha(J(N, k)) = \theta(J(N, k))$: For example, $\alpha(J(6, 2)) = 3$ but $\theta(J(6, 2)) = 4$.

But we may ask if the converse holds; does a collection of disjoint maximum cliques define a code such that each clique contains a codeword? This would then allow us to conclude, for instance, that if $\theta(J(2k, k)) = C_k$, then $C_k$ is prime. Below we will answer this in the negative in general, but note that in some cases, this is possible: One simple case occurs when there is a $j \in [N]$ such that $j$ does not belong to any clique of type A, but it does belong to every clique of type B. Then, the construction in Proposition~\ref{prop:indep-set-cover} can simply be reversed.

\begin{example}
    \label{ex:8-4-cover}
    For $N = 8$, $k = 4$, we find, with computer assistance, this is in fact all that can happen. Recall that $\theta(J(8, 4)) = 14$. Solving an appropriate ILP, one finds that any minimal clique cover must contain exactly $7$ cliques of type A and $7$ cliques of type B.
    
    For a given clique cover $\mathcal{C}$, let $\mathcal{A}(\mathcal{C}) = \{ S \mid A_S \in \mathcal{C}\}$ and $\mathcal{B}(\mathcal{C}) = \{ S \mid B_S \in \mathcal{C}\}$ be the sets underlying the cliques of type A and B respectively. By solving another ILP, one finds that any clique cover $\mathcal{C}$ for which $\bigcup_{S \in \mathcal{A}(\mathcal{C})} S = [8]$ has cardinality at least $15$ and similarly for cliques of type B. This means that for any minimal clique cover, one can find a $j \in [8]$ such the set
    \[
        \{ S \cup \{j\} \mid S \in \mathcal{A}(\mathcal{C})\} \cup \{ S \setminus \{j\} \mid S \in \mathcal{B}(\mathcal{C})\}
    \]
    is a constant-weight code with distance $4$ and with $14$ codewords.
\end{example}
\begin{example}
    For $N = 12$, $k = 6$, more work is required: The example in Table~\ref{tab:smallsolutions} shows that there is no way to get a code by simply adding/removing a single specific element from cliques as one could for $N = 8$, $k = 4$. However, the second simplest strategy still works for that particular example: For a clique $A_S$ of type A, if $10 \notin S$, then use $S \cup \{ 10 \}$. Otherwise, use $S \cup \{9 \}$. For a clique $B_S$, if $10 \in S$, use $S \setminus \{ 10\}$, otherwise use $S \setminus \{9\}$; the resulting subset of $\calP_6([12])$ is a code with distance $4$ and $132$ codewords.

    As above, let $\mathcal{C}$ denote this particular clique cover, and let $\mathcal{A}(\mathcal{C})$ denote the sets generating the cliques of type A. What makes the elements $9$ and $10$ special in that example is that
    \[
        \lvert \{ S \in \mathcal{A}(\mathcal{C}) \mid j \in S \} \rvert = \begin{cases} 15, & \text{if $j\in \{9, 10\},$} \\ 30, & \text{if $j \notin \{9, 10\}$.} \end{cases}
    \]
    One might then ask if things can get more complicated than this. As in Example~\ref{ex:8-4-cover}, by solving an appropriate ILP one finds that any minimal clique cover of $J(12, 6)$ contains exactly $66$ cliques of type A, and $66$ of type B, and we can ask how the underlying elements of $[12]$ can be distributed in the generating sets in, say, $\mathcal{A}(\mathcal{C})$, and it turns out that there are only two cases: In the simple case, corresponding to the existence of an element $j' \in [12]$ such that $j' \notin \bigcup_{S \in \mathcal{A}(\mathcal{C})} S$, we have
    \[
        a_j = \lvert \{ S \in \mathcal{A}(\mathcal{C}) \mid j \in S \} \rvert = \begin{cases} 0, & \text{if $j = j'$,} \\ 30, & \text{if $j \not= j'$,} \end{cases}
    \]
    and otherwise there exist $j_1$ and $j_2$ such that
    
    \[
        a_j = \lvert \{ S \in \mathcal{A}(\mathcal{C}) \mid j \in S \} \rvert = \begin{cases} 15, & \text{if $j\in \{j_1, j_2\}$,} \\ 30, & \text{if $j \notin \{j_1, j_2\}$,} \end{cases}
    \]
    as in the concrete example in Table~\ref{tab:smallsolutions}. To see this one solves the general set cover ILP with the additional constraint that $a_i \leq 29$, $i = 1, 2, 3$, and find that the smallest possible cover has strictly more than $132$ cliques; in fact, we find that the minimal such cover has at least $133$ and at most $135$ cliques. Similarly, one solves the original ILP with the constraint that $1 \leq a_1 \leq 14$ and find that the minimal cover has exactly $133$ elements. This means that a minimal clique cover of $J(12, 6)$ has at most two elements $j_1, j_2 \in [12]$ that appear in generators of less than $30$ cliques of type A, and since appearing in between $1$ and $14$ cliques is ruled out, the only possible options are that those two elements appear in either $0$ and $30$ cliques, or in $15$ and $15$ cliques.
\end{example}
To see that we can not in general convert a collection of disjoint maximum cliques to a corresponding code, it is convenient to rephrase the question as that of finding maximum independent sets in certain graphs: Let $JK(N, k)$ denote the graph whose vertices $V(JK(N,k)) = V(J(N, k)) \times \{0, 1\}$ correspond to two copies of the Johnson graph, and for which two vertices $(S, i)$ and $(T, j)$ are connected by an edge if either $i = j$ and $(S, T) \in E(J(N, k))$, or if $i \not= j$ and $S \cap T = \emptyset$. In other words, this is the result of adding to each part of the bipartite Kneser graph the edges from the Johnson graphs.

\begin{lem}
    Collections of disjoint maximum cliques in $J(2k, k)$ correspond to independent sets in the graph $JK(2k, k - 1)$.
\end{lem}
\begin{proof}
    The correspondence maps a clique $A_S$ of type A to $(S, 0)$ and a clique $B_S$ of type B to $(S^c, 1)$. Under this correspondence, two cliques of type A (resp. B) share an edge in $JK(2k, k-1)$ if and only if the cliques are not disjoint. Similarly, two cliques $A_S$ and $B_T$ intersect if and only if $S \subseteq T$, or equivalently, $S \cap T^c = \emptyset$.
\end{proof}
In other words, the largest collection of maximum disjoint cliques has cardinality $\alpha(JK(2k, k-1))$. On the other hand, the map $V(J(2k, k)) \to V(JK(2k-1, k-1))$ given by
\[
    S \mapsto \begin{cases} (S \setminus \{2k \}, 0), & \text{if $2k \in S$,} \\ (S^c \setminus \{2k\}, 1),  & \text {if $2k \notin S$,} \end{cases}
\]
is a graph isomorphism, so in particular, $\alpha(J(2k, k)) = \alpha(JK(2k-1, k-1)$. From this perspective, Proposition~\ref{prop:indep-set-cover} tells that $\alpha(JK(2k - 1, k-1)) \leq \alpha(JK(2k, k-1))$, and the question is when we have equality. Together with the results from \cite{Zeger2000-cm}, solving the maximum independent set problem in $JK(2k-1, k-1)$ as an ILP, we get the following result:
\begin{prop}
    For $k \leq 6$, we have $\alpha(JK(2k - 1, k-1)) = \alpha(JK(2k, k-1))$, but for $k = 7$, we have $\alpha(JK(2k-1, k-1))\leq 342$ and $\alpha(JK(2k, k-1)) \geq 371$.
\end{prop}
\begin{rem}
    An example of an independent set in $JK(14, 6)$ of cardinality $371$ may be found using the source code in the repository supplementing this work.
\end{rem}
This implies that we should also not expect that $\alpha(J(2k, k)) = C_k$ would imply that $\theta(J(2k, k)) = C_k$, but on the other hand, clearly if $\alpha(JK(2k, k-1)) = C_k$, then $\theta(J(2k, k)) = C_k$.

One simple algorithm for finding independent sets in a graph $\alpha(G)$ is the greedy algorithm which takes an ordering $v_1, \dots, v_{\lvert V(G)\rvert}$ of $V(G)$, runs through the vertices in that order, and for a given vertex, adds it to the independent set if none of the already added elements are neighbors of that vertex.

For the case of $JK(2k, k-1)$, one way to order the vertices is to first take all vertices in one half, $J(2k, k-1) \times \{0\}$, in lexicographic order, then take all elements in the other half, $J(2k, k-1) \times \{1\}$, in lexicographic order. By running the algorithm, we then find that for $k = 2, 4, 8, 16$, this construction is optimal: it produces an independent set of cardinality $C_k$. The resulting independent sets have the property that $(S, 0)$ belongs to the set if and only if $(S, 1)$ does. This means that an even simpler construction works for these cases: Order the vertices of $J(2k, k-1)$ lexicographically, then use the greedy algorithm to construct an independent set $\mathcal{W}$ in $J(2k, k-1)$. Then, check if $S_1 \cap S_2 \not= \emptyset$ for all $S_1, S_2 \in \mathcal{W}$. If this holds, $\{(S, 0) \mid S \in \mathcal{W}\} \cup \{(S, 1) \mid S \in \mathcal{W}\}$ is an independent set in $JK(2k, k-1)$. Then, if $\lvert \mathcal{W} \rvert = \tfrac{1}{2} C_k$, we have $\alpha(JK(2k, k-1)) = 2 \lvert \mathcal{W} \rvert = C_k$, and $\theta(J(2k, k)) = C_k$. In other words, since this construction is optimal for $k = 8$ and $k = 16$, this gives a computational proof of the following:
\begin{thm}
    \label{thm:theta-16-8-32-16}
    We have $\theta(J(16, 8)) = C_8 = 1430$ and $\theta(J(32, 16)) = C_{16} = 35357670$.
\end{thm}
\begin{rem}
    Since we know from \cite{Zeger2000-cm} that $\alpha(J(16, 8)) \leq 1280$, Theorem~\ref{thm:theta-16-8-32-16} gives an example where $\theta(J(2k, k)) = C_k$ but $\alpha(J(2k, k)) \not= C_k$.
\end{rem}
\begin{rem}
    The independent set in $J(2k, k-1)$ obtained with the above greedy algorithm is known as a constant-weight binary lexicographic code, or lexicode, with minimal distance $4$, so a natural question is therefore whether lexicodes with minimal distance $4$, length $2k$, constant weight $k-1$ always consist of $\tfrac{1}{2} C_k$ codewords when $k > 1$ is a power of $2$. See \cite{Conway1986-jh} for a characterization of these codes in terms of winning positions in Welter's game.
    
    Given this evidence above, we conjecture that this construction works whenever $k$ is a power of $2$.
\end{rem}
\begin{conjecture}
    Assume that $k > 1$ is a power of $2$. The lexicode in $J(2k, k-1)$ has cardinality $\tfrac{1}{2} C_k$, and its elements are pairwise overlapping. In particular, the code defines an independent set of cardinality $C_k$ in $JK(2k, k-1)$ so that $\theta(J(2k, k)) = C_k$.
\end{conjecture}
Summarizing, we know that $\theta(J(2k, k)) = C_k$ for $k = 1, 2, 4, 6, 8, 16$. For other even values of $k$ we do not have explicit values of $\theta(J(2k, k))$, leaving open the possibility that we do have covers by disjoint maximum cliques in those cases too.
\begin{question}
    Is $\theta(J(2k, k)) = C_k$ whenever $k$ is even?
\end{question}

\begin{rem}
The additional symmetry at $N = 2k$ also motivates the following question: Can we find a minimal clique cover with the property that it contains $A_S^N$ if and only if it contains $B_{S^c}^N$ and vice versa? Note that $A_S^N$ and $B_{S^c}^N$ are disjoint when $k > 1$. By the above discussion, this is possible for $k = 2, 4, 8, 16$ and using ILP solvers, we may confirm that finding such minimal clique covers is also possible for $k = 3, 6$, but for $k = 5$, the smallest such cover contains $48$ cliques, more than the $\theta(J(10, 5)) = 46$ recorded earlier.
\end{rem}

\section{Bounds from combinatorial design theory}
\label{sec:design-theory}
In the context of combinatorial design theory, following \cite{handbook}, for $v \geq s \geq t \geq 1$ and $\lambda \geq 1$, a \emph{$t$-$(v, s, \lambda)$ covering} is a pair $(X, \mathcal{B})$ where $X$ is a set with $v$ elements, and $\mathcal{B} \subseteq \calP_s(X)$ is such that every $S \in \calP_t(X)$ is a subset of at least $\lambda$ elements in $\mathcal{B}$. The \emph{covering number} $C_\lambda(v, s, t)$ is the cardinality of the smallest such $\mathcal{B}$.

Let $v \geq m \geq n$. A $(v, m, n)$ Turán system is a pair $(X, \mathcal{B})$ where $X$ has $v$ elements, $\mathcal{B} \subseteq \calP_n(X)$ such that every $S \in \calP_m(X)$ is a superset of at least one element of $\mathcal{B}$. The \emph{Turán number} $T(v, m, n)$ is the cardinality of the smallest such $\mathcal{B}$. In both cases, the elements of $\mathcal{B}$ are referred to as \emph{blocks}, and the two numbers are related through $T(v, m, n) = C_1(v, v - n, v - m)$. 

\begin{prop}
    \label{prop:cov-number-props}
    For all $N$ and $k$, we have
    \[
        \theta(J(N, k)) \leq C_1(N, k + 1, k) = T(N, N - k, N - k - 1)
    \]
    and
    \[
        \theta(J(N, k)) \leq T(N, k, k - 1) = C_1(N, N - k + 1, N - k).
    \]
    If there is a minimal cover of $J(N, k)$ consisting of only maximal cliques of type B (resp. type A), the first (resp. second) inequality is an equality. Conversely, when all minimal covers of $J(N, k)$ are mixed, the inequalities are strict.
\end{prop}
\begin{proof}
    A $k$-$(N, k + 1, 1)$ covering $([N], \mathcal{B})$ defines a clique cover $\{B_S \mid S \in \mathcal{B}\}$ of type B of $J(N, k)$. Similarly, an $(N, k, k-1)$ Turán system defines a clique cover of type A.
\end{proof}
This also allows us to pull clique covers and clique covering number bounds from the extensive literature on combinatorial designs. For explicit values of $C_1(v, s, t)$, see for example \cite{applegate2003, sidorenko1995}.

The extreme case, where a $t$-$(v, s, \lambda)$ covering $(X, \mathcal{B})$ has the property that each subset $S \in \calP_t(X)$ is a subset of \emph{exactly} $\lambda$ elements of $\mathcal{B}$ is called an $(v, s, t, \lambda)$ \emph{design}, and a $(v, s, t, 1)$ design is called a $(v, s, t)$ \emph{Steiner system}. Such systems can only exist under the ``divisibility conditions'' that $\binom{s-i}{t-i}$ divides $\lambda\binom{v-i}{t-i}$ for each $0 \leq i \leq t - 1$. The main result of \cite{keevash2019} is that for fixed $s$, $t$, and $\lambda$, $(v, s, t, \lambda)$ designs exist whenever the divisibility conditions hold for all but finitely many values of $v$.

When a $(v, s, t)$ Steiner system exists, clearly $C_1(v, s, t) = \binom{v}{t} / s$. Together with Proposition~\ref{prop:cov-number-props}, this implies the following results.

\begin{prop}
    Assume that an $(N, k + 1, k)$ Steiner system exists. Then
    \[
        \theta(J(N, k)) \leq \frac{1}{k+1}\binom{N}{k}.
    \]
\end{prop}

\begin{prop}
    Let $k \geq 1$ be fixed. There exists an $N_0 > k$ such that the following holds for all $N \geq N_0$: Assume that $k + 1 - i$ divides $\binom{N-i}{k-i}$ for all $0 \leq i \leq k - 1$. Then
    \[
        \theta(J(N, k)) \leq \frac{1}{k+1}\binom{N}{k}.
    \]
\end{prop}

\begin{rem}
One could also hope to extract exact values for $\theta(J(N, k))$ from the \cite{keevash2019} result rather than only upper bounds, but this is a taller ask: A Steiner system defines a clique cover of type B, and indeed a minimal clique cover among all clique covers of type B; however, even when a Steiner system exists, we might still be able to achieve smaller mixed clique covers. One way to rule out the existence of smaller mixed clique covers would be if those would be unable to cover all of $V(J(N, k))$. Suppose that an $(N, k+1, k)$ Steiner system exists; it then follows that $\theta(J(N, k)) = C_1(N, k+1, k) = \tfrac{1}{k+1}\binom{N}{k}$ as long as
\[
    \left(\frac{1}{k+1}\binom{N}{k} - 1\right) \omega(J(N, k)) < \binom{N}{k}.
\]
This restriction, however, implies that $k = N - 1$ which in turn just tells us that $\theta(J(N, N - 1)) = 1$ which we already knew.

As an example, take $N = 12$, $k = 5$. The bound in Theorem~\ref{thm:recursivebound} implies that $\theta(J(12, 5)) \leq 210$, but it is known that a $(12, 6, 5)$ Steiner system exists, which implies that $\theta(J(12, 5)) \leq C_1(12, 6, 5) = 132$. But from Table~\ref{tab:smallsolutions}, we read that $\theta(J(12, 5)) \leq 115$.
\end{rem}
Nevertheless, Steiner systems can still be used to define minimal clique covers. The following theorem is well-known (see \cite{schonberg} and \cite[Thm.~7]{Brouwer1990-px}):
\begin{thm}
    \label{thm:steiner-codes}
    We have
    \[
        \alpha(J(N, k)) = \frac{N(N-1) \cdots (N-k+2)}{k(k-1) \cdots 2}
    \]
    if and only if an $(N, k, k-1)$ Steiner system exists, and in that case, the blocks of the Steiner system are an independent set in $J(N, k)$.
\end{thm}

\begin{cor}
    Assume that a $(2k, k, k-1)$ Steiner system exists. Then we have $\theta(J(2k, k)) = C_k$.
\end{cor}
\begin{proof}
    By Theorem~\ref{thm:steiner-codes} and the construction in Proposition~\ref{prop:indep-set-cover}, we have a collection of disjoint maximal cliques consisting of
    \[
        \frac{N(N-1) \cdots (N-k+2)}{k(k-1) \cdots 2} = \frac{1}{\omega(J(N, k))} \binom{N}{k}
    \]
    cliques. Since all maximal cliques have the same cardinality $\omega(J(2k, k)) = k + 1$, this is a cover of $V(J(2k, k))$.
\end{proof}


\section*{Software availability}
The software used to support the computer-assisted proofs is available at\\\url{https://github.com/Kvantify/johnson-clique-cover/}.

\bibliographystyle{is-alpha}
\bibliography{references}
    \begin{table}
        \centering
        \adjustbox{max width=0.93\textwidth}{
            \begin{tblr}{r|r|X[j,valign=m]}
                $N$ & $k$ & Minimal clique cover\\
                \hline
                8 & 4 & $A^{8}_{\{{1, 2, 8}\}}$, $A^{8}_{\{{1, 3, 7}\}}$, $A^{8}_{\{{1, 4, 5}\}}$, $A^{8}_{\{{2, 3, 5}\}}$, $A^{8}_{\{{2, 4, 7}\}}$, $A^{8}_{\{{3, 4, 8}\}}$, $A^{8}_{\{{5, 7, 8}\}}$, $B^{8}_{\{{1, 2, 3, 4, 6}\}}$, $B^{8}_{\{{1, 2, 5, 6, 7}\}}$, $B^{8}_{\{{1, 3, 5, 6, 8}\}}$, $B^{8}_{\{{1, 4, 6, 7, 8}\}}$, $B^{8}_{\{{2, 3, 6, 7, 8}\}}$, $B^{8}_{\{{2, 4, 5, 6, 8}\}}$, $B^{8}_{\{{3, 4, 5, 6, 7}\}}$ \\
                \hline
                9 & 4 & $A^{9}_{\{{1, 2, 4}\}}$, $A^{9}_{\{{1, 9, 3}\}}$, $A^{9}_{\{{1, 5, 7}\}}$, $A^{9}_{\{{1, 6, 7}\}}$, $A^{9}_{\{{2, 3, 7}\}}$, $A^{9}_{\{{2, 3, 8}\}}$, $A^{9}_{\{{2, 4, 5}\}}$, $A^{9}_{\{{9, 3, 5}\}}$, $A^{9}_{\{{9, 4, 7}\}}$, $A^{9}_{\{{9, 4, 8}\}}$, $A^{9}_{\{{5, 6, 8}\}}$, $A^{9}_{\{{5, 7, 8}\}}$, $B^{9}_{\{{1, 2, 3, 5, 6}\}}$, $B^{9}_{\{{1, 2, 5, 8, 9}\}}$, $B^{9}_{\{{1, 2, 6, 8, 9}\}}$, $B^{9}_{\{{1, 2, 7, 8, 9}\}}$, $B^{9}_{\{{1, 3, 4, 5, 8}\}}$, $B^{9}_{\{{1, 3, 4, 6, 8}\}}$, $B^{9}_{\{{1, 3, 4, 7, 8}\}}$, $B^{9}_{\{{1, 4, 5, 6, 9}\}}$, $B^{9}_{\{{2, 3, 4, 6, 9}\}}$, $B^{9}_{\{{2, 4, 6, 7, 8}\}}$, $B^{9}_{\{{2, 5, 6, 7, 9}\}}$, $B^{9}_{\{{3, 4, 5, 6, 7}\}}$, $B^{9}_{\{{3, 6, 7, 8, 9}\}}$\\
                \hline
                10 & 4 & $A^{10}_{\{{1, 2, 7}\}}$, $A^{10}_{\{{1, 2, 10}\}}$, $A^{10}_{\{{1, 3, 7}\}}$, $A^{10}_{\{{1, 3, 8}\}}$, $A^{10}_{\{{1, 10, 3}\}}$, $A^{10}_{\{{1, 4, 5}\}}$, $A^{10}_{\{{1, 4, 6}\}}$, $A^{10}_{\{{1, 9, 4}\}}$, $A^{10}_{\{{1, 5, 6}\}}$, $A^{10}_{\{{1, 9, 6}\}}$, $A^{10}_{\{{1, 7, 8}\}}$, $A^{10}_{\{{1, 10, 8}\}}$, $A^{10}_{\{{2, 3, 4}\}}$, $A^{10}_{\{{2, 3, 6}\}}$, $A^{10}_{\{{2, 4, 8}\}}$, $A^{10}_{\{{2, 5, 10}\}}$, $A^{10}_{\{{2, 6, 8}\}}$, $A^{10}_{\{{9, 2, 7}\}}$, $A^{10}_{\{{3, 4, 6}\}}$, $A^{10}_{\{{3, 4, 8}\}}$, $A^{10}_{\{{3, 5, 7}\}}$, $A^{10}_{\{{9, 10, 3}\}}$, $A^{10}_{\{{9, 4, 5}\}}$, $A^{10}_{\{{10, 4, 5}\}}$, $A^{10}_{\{{4, 6, 8}\}}$, $A^{10}_{\{{9, 4, 7}\}}$, $A^{10}_{\{{10, 4, 7}\}}$, $A^{10}_{\{{9, 5, 6}\}}$, $A^{10}_{\{{5, 7, 8}\}}$, $A^{10}_{\{{10, 6, 7}\}}$, $A^{10}_{\{{9, 10, 8}\}}$, $B^{10}_{\{{1, 2, 3, 5, 9}\}}$, $B^{10}_{\{{1, 2, 5, 8, 9}\}}$, $B^{10}_{\{{1, 5, 7, 9, 10}\}}$, $B^{10}_{\{{2, 3, 5, 8, 9}\}}$, $B^{10}_{\{{2, 3, 7, 8, 10}\}}$, $B^{10}_{\{{2, 4, 5, 6, 7}\}}$, $B^{10}_{\{{2, 4, 6, 9, 10}\}}$, $B^{10}_{\{{3, 5, 6, 8, 10}\}}$, $B^{10}_{\{{3, 6, 7, 8, 9}\}}$ \\
                \hline
                10 & 5 & $A^{10}_{\{{1, 2, 3, 4}\}}$, $A^{10}_{\{{1, 2, 5, 7}\}}$, $A^{10}_{\{{1, 3, 5, 8}\}}$, $A^{10}_{\{{1, 3, 6, 7}\}}$, $A^{10}_{\{{1, 10, 3, 6}\}}$, $A^{10}_{\{{1, 9, 3, 7}\}}$, $A^{10}_{\{{1, 4, 7, 8}\}}$, $A^{10}_{\{{1, 10, 6, 7}\}}$, $A^{10}_{\{{2, 3, 5, 6}\}}$, $A^{10}_{\{{2, 4, 5, 8}\}}$, $A^{10}_{\{{2, 4, 5, 10}\}}$, $A^{10}_{\{{2, 4, 6, 7}\}}$, $A^{10}_{\{{9, 2, 4, 8}\}}$, $A^{10}_{\{{9, 2, 5, 8}\}}$, $A^{10}_{\{{3, 4, 6, 8}\}}$, $A^{10}_{\{{9, 3, 6, 7}\}}$, $A^{10}_{\{{10, 4, 5, 8}\}}$, $A^{10}_{\{{5, 6, 7, 8}\}}$, $B^{10}_{\{{1, 2, 3, 5, 9, 10}\}}$, $B^{10}_{\{{1, 2, 3, 6, 8, 9}\}}$, $B^{10}_{\{{1, 2, 3, 7, 8, 10}\}}$, $B^{10}_{\{{1, 2, 4, 5, 6, 9}\}}$, $B^{10}_{\{{1, 2, 4, 6, 8, 10}\}}$, $B^{10}_{\{{1, 2, 4, 7, 9, 10}\}}$, $B^{10}_{\{{1, 2, 5, 6, 8, 10}\}}$, $B^{10}_{\{{1, 2, 6, 7, 8, 9}\}}$, $B^{10}_{\{{1, 2, 6, 8, 9, 10}\}}$, $B^{10}_{\{{1, 3, 4, 5, 6, 9}\}}$, $B^{10}_{\{{1, 3, 4, 5, 7, 10}\}}$, $B^{10}_{\{{1, 3, 4, 8, 9, 10}\}}$, $B^{10}_{\{{1, 4, 5, 6, 7, 9}\}}$, $B^{10}_{\{{1, 4, 5, 6, 8, 9}\}}$, $B^{10}_{\{{1, 4, 5, 6, 9, 10}\}}$, $B^{10}_{\{{1, 5, 7, 8, 9, 10}\}}$, $B^{10}_{\{{2, 3, 4, 5, 7, 9}\}}$, $B^{10}_{\{{2, 3, 4, 6, 9, 10}\}}$, $B^{10}_{\{{2, 3, 4, 7, 8, 10}\}}$, $B^{10}_{\{{2, 3, 5, 7, 8, 10}\}}$, $B^{10}_{\{{2, 3, 6, 7, 8, 10}\}}$, $B^{10}_{\{{2, 3, 7, 8, 9, 10}\}}$, $B^{10}_{\{{2, 5, 6, 7, 9, 10}\}}$, $B^{10}_{\{{3, 4, 5, 6, 7, 10}\}}$, $B^{10}_{\{{3, 4, 5, 7, 8, 9}\}}$, $B^{10}_{\{{3, 4, 5, 7, 9, 10}\}}$, $B^{10}_{\{{3, 5, 6, 8, 9, 10}\}}$, $B^{10}_{\{{4, 6, 7, 8, 9, 10}\}}$ \\
                \hline
                11 & 4 & $A^{11}_{\{{1, 2, 9}\}}$, $A^{11}_{\{{1, 2, 10}\}}$, $A^{11}_{\{{1, 3, 6}\}}$, $A^{11}_{\{{1, 3, 7}\}}$, $A^{11}_{\{{1, 4, 5}\}}$, $A^{11}_{\{{1, 4, 6}\}}$, $A^{11}_{\{{1, 5, 8}\}}$, $A^{11}_{\{{1, 7, 8}\}}$, $A^{11}_{\{{1, 9, 10}\}}$, $A^{11}_{\{{1, 9, 11}\}}$, $A^{11}_{\{{1, 10, 11}\}}$, $A^{11}_{\{{2, 3, 4}\}}$, $A^{11}_{\{{2, 3, 8}\}}$, $A^{11}_{\{{2, 4, 8}\}}$, $A^{11}_{\{{2, 5, 6}\}}$, $A^{11}_{\{{2, 5, 7}\}}$, $A^{11}_{\{{2, 6, 7}\}}$, $A^{11}_{\{{9, 2, 10}\}}$, $A^{11}_{\{{4, 3, 11}\}}$, $A^{11}_{\{{9, 3, 5}\}}$, $A^{11}_{\{{10, 3, 5}\}}$, $A^{11}_{\{{10, 3, 6}\}}$, $A^{11}_{\{{9, 3, 7}\}}$, $A^{11}_{\{{3, 11, 8}\}}$, $A^{11}_{\{{10, 4, 5}\}}$, $A^{11}_{\{{9, 4, 6}\}}$, $A^{11}_{\{{9, 4, 7}\}}$, $A^{11}_{\{{10, 4, 7}\}}$, $A^{11}_{\{{11, 4, 8}\}}$, $A^{11}_{\{{5, 6, 7}\}}$, $A^{11}_{\{{11, 5, 6}\}}$, $A^{11}_{\{{11, 5, 7}\}}$, $A^{11}_{\{{9, 5, 8}\}}$, $A^{11}_{\{{11, 6, 7}\}}$, $A^{11}_{\{{9, 6, 8}\}}$, $A^{11}_{\{{10, 6, 8}\}}$, $A^{11}_{\{{10, 7, 8}\}}$, $A^{11}_{\{{9, 10, 11}\}}$, $B^{11}_{\{{1, 2, 3, 5, 11}\}}$, $B^{11}_{\{{1, 2, 4, 7, 11}\}}$, $B^{11}_{\{{1, 2, 6, 8, 11}\}}$, $B^{11}_{\{{1, 3, 4, 8, 9}\}}$, $B^{11}_{\{{1, 3, 4, 8, 10}\}}$, $B^{11}_{\{{1, 5, 6, 9, 10}\}}$, $B^{11}_{\{{1, 5, 7, 9, 10}\}}$, $B^{11}_{\{{1, 6, 7, 9, 10}\}}$, $B^{11}_{\{{2, 3, 6, 9, 11}\}}$, $B^{11}_{\{{2, 3, 7, 10, 11}\}}$, $B^{11}_{\{{2, 4, 5, 9, 11}\}}$, $B^{11}_{\{{2, 4, 6, 10, 11}\}}$, $B^{11}_{\{{2, 5, 8, 10, 11}\}}$, $B^{11}_{\{{2, 7, 8, 9, 11}\}}$, $B^{11}_{\{{3, 4, 5, 6, 8}\}}$, $B^{11}_{\{{3, 4, 5, 7, 8}\}}$, $B^{11}_{\{{3, 4, 6, 7, 8}\}}$, $B^{11}_{\{{3, 4, 8, 9, 10}\}}$ \\
                \hline
                12 & 6 & $A^{12}_{\{{1, 2, 3, 4, 5}\}}$, $A^{12}_{\{{1, 2, 3, 6, 12}\}}$, $A^{12}_{\{{1, 2, 3, 7, 8}\}}$, $A^{12}_{\{{1, 2, 3, 9, 11}\}}$, $A^{12}_{\{{1, 2, 4, 6, 10}\}}$, $A^{12}_{\{{1, 2, 4, 7, 9}\}}$, $A^{12}_{\{{1, 2, 4, 8, 12}\}}$, $A^{12}_{\{{1, 2, 5, 6, 7}\}}$, $A^{12}_{\{{1, 2, 5, 8, 10}\}}$, $A^{12}_{\{{1, 2, 5, 9, 12}\}}$, $A^{12}_{\{{1, 2, 6, 8, 9}\}}$, $A^{12}_{\{{1, 2, 7, 11, 12}\}}$, $A^{12}_{\{{1, 3, 4, 6, 8}\}}$, $A^{12}_{\{{1, 3, 4, 7, 10}\}}$, $A^{12}_{\{{1, 3, 4, 9, 12}\}}$, $A^{12}_{\{{1, 3, 5, 6, 11}\}}$, $A^{12}_{\{{1, 3, 5, 7, 12}\}}$, $A^{12}_{\{{1, 3, 5, 8, 9}\}}$, $A^{12}_{\{{1, 3, 6, 7, 9}\}}$, $A^{12}_{\{{1, 3, 8, 10, 12}\}}$, $A^{12}_{\{{1, 4, 5, 6, 9}\}}$, $A^{12}_{\{{1, 4, 5, 7, 8}\}}$, $A^{12}_{\{{1, 4, 5, 11, 12}\}}$, $A^{12}_{\{{1, 4, 6, 7, 12}\}}$, $A^{12}_{\{{1, 4, 8, 9, 11}\}}$, $A^{12}_{\{{1, 5, 6, 8, 12}\}}$, $A^{12}_{\{{1, 5, 7, 9, 10}\}}$, $A^{12}_{\{{1, 6, 7, 8, 11}\}}$, $A^{12}_{\{{1, 6, 9, 10, 12}\}}$, $A^{12}_{\{{1, 7, 8, 9, 12}\}}$, $A^{12}_{\{{2, 3, 4, 6, 9}\}}$, $A^{12}_{\{{2, 3, 4, 7, 12}\}}$, $A^{12}_{\{{2, 3, 4, 8, 11}\}}$, $A^{12}_{\{{2, 3, 5, 6, 8}\}}$, $A^{12}_{\{{2, 3, 5, 7, 9}\}}$, $A^{12}_{\{{2, 3, 5, 10, 12}\}}$, $A^{12}_{\{{2, 3, 6, 7, 10}\}}$, $A^{12}_{\{{2, 3, 8, 9, 12}\}}$, $A^{12}_{\{{2, 4, 5, 6, 12}\}}$, $A^{12}_{\{{2, 4, 5, 7, 11}\}}$, $A^{12}_{\{{2, 4, 5, 8, 9}\}}$, $A^{12}_{\{{2, 4, 6, 7, 8}\}}$, $A^{12}_{\{{2, 4, 9, 10, 12}\}}$, $A^{12}_{\{{2, 5, 6, 9, 11}\}}$, $A^{12}_{\{{2, 5, 7, 8, 12}\}}$, $A^{12}_{\{{2, 6, 7, 9, 12}\}}$, $A^{12}_{\{{2, 6, 8, 11, 12}\}}$, $A^{12}_{\{{2, 7, 8, 9, 10}\}}$, $A^{12}_{\{{3, 4, 5, 6, 7}\}}$, $A^{12}_{\{{3, 4, 5, 8, 12}\}}$, $A^{12}_{\{{3, 4, 5, 9, 10}\}}$, $A^{12}_{\{{3, 4, 6, 11, 12}\}}$, $A^{12}_{\{{3, 4, 7, 8, 9}\}}$, $A^{12}_{\{{3, 5, 6, 9, 12}\}}$, $A^{12}_{\{{3, 5, 7, 8, 11}\}}$, $A^{12}_{\{{3, 6, 7, 8, 12}\}}$, $A^{12}_{\{{3, 6, 8, 9, 10}\}}$, $A^{12}_{\{{3, 7, 9, 11, 12}\}}$, $A^{12}_{\{{4, 5, 6, 8, 10}\}}$, $A^{12}_{\{{4, 5, 7, 9, 12}\}}$, $A^{12}_{\{{4, 6, 7, 9, 11}\}}$, $A^{12}_{\{{4, 6, 8, 9, 12}\}}$, $A^{12}_{\{{4, 7, 8, 10, 12}\}}$, $A^{12}_{\{{5, 6, 7, 8, 9}\}}$, $A^{12}_{\{{5, 6, 7, 10, 12}\}}$, $A^{12}_{\{{5, 8, 9, 11, 12}\}}$, $B^{12}_{\{{1, 2, 3, 4, 6, 7, 11}\}}$, $B^{12}_{\{{1, 2, 3, 4, 8, 9, 10}\}}$, $B^{12}_{\{{1, 2, 3, 4, 10, 11, 12}\}}$, $B^{12}_{\{{1, 2, 3, 5, 6, 9, 10}\}}$, $B^{12}_{\{{1, 2, 3, 5, 7, 10, 11}\}}$, $B^{12}_{\{{1, 2, 3, 5, 8, 11, 12}\}}$, $B^{12}_{\{{1, 2, 3, 6, 8, 10, 11}\}}$, $B^{12}_{\{{1, 2, 3, 7, 9, 10, 12}\}}$, $B^{12}_{\{{1, 2, 4, 5, 6, 8, 11}\}}$, $B^{12}_{\{{1, 2, 4, 5, 7, 10, 12}\}}$, $B^{12}_{\{{1, 2, 4, 5, 9, 10, 11}\}}$, $B^{12}_{\{{1, 2, 4, 6, 9, 11, 12}\}}$, $B^{12}_{\{{1, 2, 4, 7, 8, 10, 11}\}}$, $B^{12}_{\{{1, 2, 5, 6, 10, 11, 12}\}}$, $B^{12}_{\{{1, 2, 5, 7, 8, 9, 11}\}}$, $B^{12}_{\{{1, 2, 6, 7, 8, 10, 12}\}}$, $B^{12}_{\{{1, 2, 6, 7, 9, 10, 11}\}}$, $B^{12}_{\{{1, 2, 8, 9, 10, 11, 12}\}}$, $B^{12}_{\{{1, 3, 4, 5, 6, 10, 12}\}}$, $B^{12}_{\{{1, 3, 4, 5, 7, 9, 11}\}}$, $B^{12}_{\{{1, 3, 4, 5, 8, 10, 11}\}}$, $B^{12}_{\{{1, 3, 4, 6, 9, 10, 11}\}}$, $B^{12}_{\{{1, 3, 4, 7, 8, 11, 12}\}}$, $B^{12}_{\{{1, 3, 5, 6, 7, 8, 10}\}}$, $B^{12}_{\{{1, 3, 5, 9, 10, 11, 12}\}}$, $B^{12}_{\{{1, 3, 6, 7, 10, 11, 12}\}}$, $B^{12}_{\{{1, 3, 6, 8, 9, 11, 12}\}}$, $B^{12}_{\{{1, 3, 7, 8, 9, 10, 11}\}}$, $B^{12}_{\{{1, 4, 5, 6, 7, 10, 11}\}}$, $B^{12}_{\{{1, 4, 5, 8, 9, 10, 12}\}}$, $B^{12}_{\{{1, 4, 6, 7, 8, 9, 10}\}}$, $B^{12}_{\{{1, 4, 6, 8, 10, 11, 12}\}}$, $B^{12}_{\{{1, 4, 7, 9, 10, 11, 12}\}}$, $B^{12}_{\{{1, 5, 6, 7, 9, 11, 12}\}}$, $B^{12}_{\{{1, 5, 6, 8, 9, 10, 11}\}}$, $B^{12}_{\{{1, 5, 7, 8, 10, 11, 12}\}}$, $B^{12}_{\{{2, 3, 4, 5, 6, 10, 11}\}}$, $B^{12}_{\{{2, 3, 4, 5, 7, 8, 10}\}}$, $B^{12}_{\{{2, 3, 4, 5, 9, 11, 12}\}}$, $B^{12}_{\{{2, 3, 4, 6, 8, 10, 12}\}}$, $B^{12}_{\{{2, 3, 4, 7, 9, 10, 11}\}}$, $B^{12}_{\{{2, 3, 5, 6, 7, 11, 12}\}}$, $B^{12}_{\{{2, 3, 5, 8, 9, 10, 11}\}}$, $B^{12}_{\{{2, 3, 6, 7, 8, 9, 11}\}}$, $B^{12}_{\{{2, 3, 6, 9, 10, 11, 12}\}}$, $B^{12}_{\{{2, 3, 7, 8, 10, 11, 12}\}}$, $B^{12}_{\{{2, 4, 5, 6, 7, 9, 10}\}}$, $B^{12}_{\{{2, 4, 5, 8, 10, 11, 12}\}}$, $B^{12}_{\{{2, 4, 6, 7, 10, 11, 12}\}}$, $B^{12}_{\{{2, 4, 6, 8, 9, 10, 11}\}}$, $B^{12}_{\{{2, 4, 7, 8, 9, 11, 12}\}}$, $B^{12}_{\{{2, 5, 6, 7, 8, 10, 11}\}}$, $B^{12}_{\{{2, 5, 6, 8, 9, 10, 12}\}}$, $B^{12}_{\{{2, 5, 7, 9, 10, 11, 12}\}}$, $B^{12}_{\{{3, 4, 5, 6, 8, 9, 11}\}}$, $B^{12}_{\{{3, 4, 5, 7, 10, 11, 12}\}}$, $B^{12}_{\{{3, 4, 6, 7, 8, 10, 11}\}}$, $B^{12}_{\{{3, 4, 6, 7, 9, 10, 12}\}}$, $B^{12}_{\{{3, 4, 8, 9, 10, 11, 12}\}}$, $B^{12}_{\{{3, 5, 6, 7, 9, 10, 11}\}}$, $B^{12}_{\{{3, 5, 6, 8, 10, 11, 12}\}}$, $B^{12}_{\{{3, 5, 7, 8, 9, 10, 12}\}}$, $B^{12}_{\{{4, 5, 6, 7, 8, 11, 12}\}}$, $B^{12}_{\{{4, 5, 6, 9, 10, 11, 12}\}}$, $B^{12}_{\{{4, 5, 7, 8, 9, 10, 11}\}}$, $B^{12}_{\{{6, 7, 8, 9, 10, 11, 12}\}}$
         \\
        \end{tblr}
        }
        \caption{Minimal clique covers for small $N$ and $k \geq 4$.}
        \label{tab:smallsolutions}
    \end{table}
\end{document}